\documentclass[12pt]{article}
\usepackage{fullpage}
\usepackage{amssymb,amsmath,amsfonts,amsthm,latexsym,enumerate,url,cases}
\numberwithin{equation}{section}
\usepackage{hyperref}
\hypersetup{colorlinks=true,citecolor=blue,linkcolor=blue,urlcolor=blue}

\newtheorem{theorem}{Theorem}[section] %
\newtheorem{lemma}[theorem]{Lemma} %
\newtheorem{corollary}[theorem]{Corollary} %
\newtheorem{problem}[theorem]{Problem} %
\usepackage[left=50pt,right=50pt,top=30pt,bottom=80pt]{geometry}
\begin{document}
\title{A note on the lower bound of representation functions}
\author{{Xing-Wang Jiang\footnote{xwjiangnj@sina.com(X.-W. Jiang)}, Csaba S\'{a}ndor\footnote{csandor@math.bme.hu(C. S\'{a}ndor)}, Quan-Hui Yang\footnote{yangquanhui01@163.com(Q.-H. Yang)} }\\
\small  *School of Mathematical Sciences and Institute of Mathematics,\\\small Nanjing Normal University, Nanjing 210023, China\\
\small  $\dagger$Institute of Mathematics, Budapest University of Technology and Economics,\\
\small MTA-BME Lend\" ulet
Arithmetic Combinatorics Research Group\\
\small  H-1529 B. O. Box, Hungary\\
\small  $\ddagger$School of Mathematics and Statistics, Nanjing University of Information \\\small Science and Technology,
Nanjing 210044, China
}
\date{}
\maketitle \baselineskip 18pt \maketitle \baselineskip 18pt

{\bf Abstract.}
For a set $A$ of nonnegative integers, let $R_2(A,n)$ denote the number of solutions to $n=a+a'$ with $a,a'\in A$, $a<a'$. Let $A_0$ be the Thue-Morse sequence and $B_0=\mathbb{N}\setminus A_0$. Let $A\subset \mathbb{N}$ and $N$ be a positive integer such that $R_2(A,n)=R_2(\mathbb{N}\setminus A,n)$ for all $n\geq 2N-1$. Previously, the first author proved that if $|A\cap A_0|=+\infty$ and $|A\cap B_0|=+\infty$, then $R_2(A,n)\geq \frac{n+3}{56N-52}-1$ for all $n\geq 1$. In this paper, we prove that the above lower bound is nearly best possible.
We also get some other results.

 \vskip 3mm
 {\bf 2010 Mathematics Subject Classification:} 11B34, 11B83

 {\bf Keywords and phrases:} partition, representation function, Thue-Morse sequence, S\'{a}rk\"{o}zy's problem
\vskip 5mm

\section{Introduction}
Let $\mathbb{N}$ be the set of nonnegative integers. For a set $A\subseteq\mathbb{N}$, let $R_2(A,n)$ denote the number of solutions to $a+a'=n$ with $a,a'\in A$, $a<a'$.
In the last few decades, there are many research on the representation function $R_2(A,n)$. Professor S\'{a}rk\"{o}zy posed the problem that whether there exist two sets $A$ and $B$ with infinite symmetric difference such that $R_2(A,n)=R_2(B,n)$ for all sufficiently large integers $n$. Let $D(0)=0$ and let $D(a)$ be the number of ones in the binary representation of $a$. Let $A_0$ be the set of all nonnegative integers $a$ with even $D(a)$ and let $B_0=\mathbb{N}\setminus A_0$. The sequence $A_0$ is called Thue-Morse sequence. In 2002, Dombi \cite{D} proved that $R_2(A_0,n)=R_2(B_0,n)$ for all $n\geq0$ which answered the problem. In 2004, using the generating functions, S\'{a}ndor \cite{S} proved the following precise formulation.
\vskip 2mm
\noindent\text{Theorem A} (\cite[Theorem 1]{S}).
{\it Let $N$ be a positive integer. Then $R_2(A,n)=R_2(\mathbb{N}\setminus A,n)$ for all $n\geq 2N-1$ if and only if $|A\cap[0,2N-1]|=N$ and $2m\in A\Leftrightarrow m\in A,~2m+1\in A\Leftrightarrow m\notin A$ for all $m\geq N$.}
\vskip 2mm
Let $R_{A,B}(n)$ be the number of solutions to
$a+b=n,a\in A,b\in B.$ In 2011, Chen \cite{C} gave the lower bound of $R_2(A,n)$. The following theorems are proved.
\vskip 2mm
\noindent\text{Theorem B} (\cite[Theorem 1.4 (i)]{C}).
{\it Let $A$ be a subset of $\mathbb{N}$ and $N$ be a positive integer such that $R_2(A,n)=R_2(\mathbb{N}\setminus A,n)$ for all $n\geq 2N-1$. If $|A\cap A_0|=+\infty$ and $|A\cap B_0|=+\infty$, then for all $n\geq1$, we have
$$R_2(A,n)\geq\frac{n}{40N(N+1)}-1,~~R_{A,\mathbb{N}\setminus A}(n)\geq \frac{n}{20N(N+1)}-1.$$}

\vskip -4mm
\noindent\text{Theorem C} (\cite[Theorem 1.5]{C}).
{\it Let $A$ be a subset of $\mathbb{N}$ and $N$ be a positive integer such that $R_2(A,n)=R_2(\mathbb{N}\setminus A,n)$ for all $n\geq 2N-1$. Then, for any function $f$ with $f(x)\rightarrow +\infty$ as $x\rightarrow +\infty$, the set of integers $n$ with
$$R_2(A,n)\geq \frac{n}{16}-f(n),~R_{A,\mathbb{N}\setminus A}(n)\geq \frac{n}{8}-f(n)$$
has the density one.}
\vskip 2mm
Recently, Jiang, S\'{a}ndor and Yang \cite{JSY} improved the result of Theorem C as follows.

\vskip 2mm
\noindent\text{Theorem D} (\cite[Corrolary 1.3]{JSY}).
{\it Let $A$ be a subset of $\mathbb{N}$ and $N$ be a positive integer such that $R_2(A,n)=R_2(\mathbb{N}\setminus A,n)$ for all $n\geq 2N-1$. Then for any $0<\theta<\frac{2\log2-\log3}{42\log 2-9\log3}=0.0149\dots$, the set of integers $n$ with
$$R_2(A,n)=\frac{n}{8}+O(n^{1-\theta})$$
has density one.}
\vskip 2mm
They also posed a problem for further research.
\begin{problem}(\cite[Problem 1.5]{JSY})
Let $A$ be a subset of $\mathbb{N}$ and $N$ be a positive integer such that $R_2(A,n)=R_2(\mathbb{N}\setminus A,n)$ for all $n\geq 2N-1$. Does there exist a sequence $n_1,n_2,\ldots$ such that
$$\lim_{k\rightarrow\infty}\frac{R_2(A,n_k)}{n_k}\neq\frac 1 8?$$
\end{problem}

Afterwards, Jiang \cite{J1} improved the result of Theorem B.

\vskip 2mm
\noindent\text{Theorem E} (\cite[Theorem 1.4]{J1}).
{\it Let $A$ be a subset of $\mathbb{N}$ and $N\geq 2$ such that $R_2(A,n)=R_2(\mathbb{N}\setminus A,n)$ for all $n\geq 2N-1$. If $|A\cap A_0|=+\infty$ and $|A\cap B_0|=+\infty$, then for all $n\geq1$, we have
$$R_2(A,n)\geq\frac{n+3}{56N-52}-1,~~R_{A,\mathbb{N}\setminus A}(n)\geq \frac{n+3}{28N-26}-1.$$}

For more related result, see \cite{CL,CSST,CW,CT,ES1,ES2,ESS1,J,T,TY,YT}. In this paper, we prove that the lower bound of Theorem E is nearly
best possible. The following theorems are proved.
\begin{theorem}\label{thm1}
Let $N\geq2$. Then there exists a set $A\subseteq \mathbb{N}$ such that $R_2(A,n)=R_2(\mathbb{N}\setminus A,n)$ for all $n\geq 2N-1$, $|A\cap A_0|=+\infty$, $|A\cap B_0|=+\infty$ and
$$R_2(A,n)\leq \frac{n+1}{4\cdot 2^{\lfloor \log _2(N-1) \rfloor }}$$
for infinitely many integers $n$.
\end{theorem}

As corollaries we get the following two results

\begin{corollary}
Let $N\geq3$. Then there exists a set $A\subseteq \mathbb{N}$ such that $R_2(A,n)=R_2(\mathbb{N}\setminus A,n)$ for all $n\geq 2N-1$, $|A\cap A_0|=+\infty$, $|A\cap B_0|=+\infty$ and
$$R_2(A,n)< \frac{n+1}{2(N-1)}$$
for infinitely many integers $n$.
\end{corollary}

\begin{corollary}
For $k\geq0$, let $M_k=2^k+1$. Then there exists a set $D_k\subseteq \mathbb{N}$ such that $R_2(D_k,n)=R_2(\mathbb{N}\setminus D_k,n)$ for all $n\geq 2M_k-1$, $|D_k\cap A_0|=+\infty$, $|D_k\cap B_0|=+\infty$ and
$$R_2(D_k,n)\leq \frac{n+1}{4(M_k-1)}$$
for infinitely many integers $n$.
\end{corollary}

A trivial corollary of \cite{J1} is the following result.

\begin{corollary}
Let $A$ be a subset of $\mathbb{N}$ and $N\geq 2$ such that $R_2(A,n)=R_2(\mathbb{N}\setminus A,n)$ for all $n\geq 2N-1$. If $|A\cap A_0|=+\infty$ and $|A\cap B_0|=+\infty$, then for all $n\geq56N-54$, we have
$$R_2(A,n)\geq 1.$$
\end{corollary}

The bound $56N-54$ can be improved to $14(N-1)$ (see \cite{J1}).

\vskip 2mm
\noindent\text{Theorem F} (\cite[Theorem 1.3]{J1}).
{Let $A$ be a subset of $\mathbb{N}$ and $N\geq 2$ such that $R_2(A,n)=R_2(\mathbb{N}\setminus A,n)$ for all $n\geq 2N-1$. If $A\ne A_0,B_0$ then for all $n\geq14(N-1)$, we have
$$R_2(A,n)\geq 1.$$}

The following statements show that the previous theorem nearly best possible.
\begin{theorem}\label{thm2}
For nonnegative integers $k$, let $N_k=3\cdot4^k$. Then there exists a set $C_k\subseteq \mathbb{N}$ such that
$R_2(C_k,n)=R_2(\mathbb{N}\setminus C_k,n)$ for all $n\geq 2N_k-1$, $|C_k\cap A_0|=+\infty$, $|C_k\cap B_0|=+\infty$ and
$$R_2\left(C_k,\frac{14}{3}N_k-1\right)=0.$$
\end{theorem}

\begin{theorem}\label{thm3}
Let $N\geq3$. Then there exists a set $A\subseteq \mathbb{N}$ such that $R_2(A,n)=R_2(\mathbb{N}\setminus A,n)$ for all $n\geq 2N-1$, $|A\cap A_0|=+\infty$, $|A\cap B_0|=+\infty$ and
$$R_2(A,3N-1)=0.$$
\end{theorem}

\section{Proofs}
In this section, we prove our main results. Firstly, we give a lemma which will be used in the proofs of theorems.

\begin{lemma}(\cite[Lemma 1]{CT})\label{lem1}
Let $A$ be a subset of $\mathbb{N}$ such that $R_2(A,n)=R_2(\mathbb{N}\setminus A,n)$ for all $n\geq 2N-1$. Let $l,k,i$ be integers with $k\geq N,~l\geq 0$ and $0\leq i\leq 2^l-1$. Then

(i) if $2|D(i)$, then $k\in A\Leftrightarrow 2^lk+i\in A$;

(ii) if $2\nmid D(i)$, then $k\in A\Leftrightarrow 2^lk+i \notin A.$
\end{lemma}

Now, we give the proofs of theorems.
\begin{proof}[Proof of Theorem \ref{thm1}]
Let
\begin{eqnarray}\label{eq1}
A\cap[0,2N-1]=(A_0\cap[2,2N-3])\cup\{2N-2,2N-1\}.
\end{eqnarray}
For $a\geq N$, let $2a\in A$, $2a+1\notin A$ if $a\in A$ and $2a\notin A$, $2a+1\in A$ if $a\notin A$. By Theorem A, $R_2(A,n)=R_2(\mathbb{N}\setminus A,n)$ for all $n\geq 2N-1$. By definition of $A_0$, we get that $|\{2N-2,2N-1\}\cap A_0|=1$. It follows that $|A\cap A_0|=+\infty$ and $|A\cap B_0|=+\infty$. It is enough to prove that there are infinitely many integers $n$ with $R_2(A,n)\leq \frac{n+1}{4\cdot 2^{\lfloor \log _2(N-1) \rfloor }}$.

For $N\leq k\leq 2N-1$, let $S^{(k)}=A_0$ if $k\in A$ and $S^{(k)}=B_0$ if $k\notin A$. By Lemma \ref{lem1}, one can get easily that
\begin{eqnarray*}
A=(A_0\cap[2,N-1])\cup(\cup_{l=0}^{\infty}\cup_{k=N}^{2N-1}(2^lk+[0,2^l-1]\cap S^{(k)})).
\end{eqnarray*}
Furthermore, by (\ref{eq1}),
$$
(\cup_{l=0}^{\infty}\cup_{k=N}^{2N-3}(2^lk+[0,2^l-1]\cap S^{(k)}))\subseteq A_0
$$
and we have that, if $2N-2\in B_0$, then
\begin{eqnarray*}
&(\cup_{l=0}^{\infty}(2^l(2N-2)+[0,2^l-1]\cap S^{(2N-2)}))\subseteq B_0,\\
&(\cup_{l=0}^{\infty}(2^l(2N-1)+[0,2^l-1]\cap S^{(2N-1)}))\subseteq A_0;
\end{eqnarray*}
if $2N-1\in B_0$, then
\begin{eqnarray*}
&(\cup_{l=0}^{\infty}(2^l(2N-2)+[0,2^l-1]\cap S^{(2N-2)}))\subseteq A_0,\\
&(\cup_{l=0}^{\infty}(2^l(2N-1)+[0,2^l-1]\cap S^{(2N-1)}))\subseteq B_0.
\end{eqnarray*}

It is known that $R_2(A_0,2^{2m+1}-1)=0$ for any nonnegative integers $m$. Thus, if $2^{2m+1}-1=a+a'$ with $a<a'$ and $a,a'\in A$, then either $a\in A\setminus A_0$ or $a'\in A\setminus A_0$. Let us suppose that $2N-i\in B_0$, where $i=1$ or $2$. Then
$$
A\setminus A_0=\cup_{l=0}^{\infty}(2^l(2N-i)+[0,2^l-1]\cap A_0).
$$
Hence
\begin{eqnarray*}
R_2(A,2^{2m+1}-1)&\leq&|(A\setminus A_0)\cap[0,2^{2m+1}-1]|\\
&=&|(\cup_{l=0}^{\infty}(2^l(2N-i)+[0,2^l-1]\cap A_0))\cap[0,2^{2m+1}-1]|.
\end{eqnarray*}
If $2^l(2N-2)\geq 2^{2m+1}$ for some $l$, then for
$$a\in(2^l(2N-i)+[0,2^l-1]\cap A_0),$$
we have $a\geq 2^{2m+1}$. Therefore, for $m> \frac{\log_2(N-1)}{2}$,
\begin{eqnarray*}
R_2(A,2^{2m+1}-1)\leq\sum_{l=0}^{\lceil2m-1-\log_2(N-1)\rceil}|[0,2^l-1]\cap A_0|\\
=1+\sum_{l=1}^{\lceil2m-1-\log_2(N-1)\rceil}2^{l-1}=2^{\lceil2m-1-\log_2(N-1)\rceil}= \frac{2^{2m+1}}{4\cdot 2^{\lfloor \log _2(N-1)\rfloor}}.
\end{eqnarray*}
This completes the proof.
\end{proof}

\begin{proof}[Proof of Theorem \ref{thm2}]
Let
$$C_0\cap[0,5]=\{0,3,4\},$$
and let $2c\in C_0,~2c+1\notin C_0$ if $c\in C_0$ and $2c\notin C_0,~2c+1\in C_0$ if $c\notin C_0$ for $c\geq 3$.
For $k\geq1$, let
\begin{eqnarray*}
C_k=(4C_{k-1}+\{0,3\})\cup(4(\mathbb{N}\setminus C_{k-1})+\{1,2\}).
\end{eqnarray*}
Now, we prove that $C_k$ is the desired set by using induction on $k$.

By Theorem A, one can verify easily that Theorem \ref{thm2} holds for $k=0$. Suppose that it is true for $k-1~(k\geq1)$. Then
$$|C_{k-1}\cap[0,2N_{k-1}-1]|=|(\mathbb{N}\setminus C_{k-1})\cap[0,2N_{k-1}-1]|=N_{k-1},$$
$$|C_{k-1}\cap A_0|=+\infty,~~|C_{k-1}\cap B_0|=+\infty$$
and for $m\geq N_{k-1}$,
\begin{eqnarray}\label{eq2}
2m\in C_{k-1}\Leftrightarrow m\in C_{k-1},~~2m+1\in C_{k-1}\Leftrightarrow m\notin C_{k-1}.
\end{eqnarray}
By the definition of $C_k$,
$$|C_k\cap[0,2N_k-1]|=2(|C_{k-1}\cap[0,2N_{k-1}-1]|+|(\mathbb{N}\setminus C_{k-1})\cap[0,2N_{k-1}-1]|)=N_k$$
and
$$|C_{k}\cap A_0|=+\infty,~~|C_{k}\cap B_0|=+\infty.$$
For $m\geq N_k$, by (\ref{eq2}),
\begin{eqnarray*}
2m\in C_k\Leftrightarrow m\in C_k,~~2m+1\in C_k\Leftrightarrow m\notin C_k.
\end{eqnarray*}
It follows from Theorem A that $R_2(C_k,n)=R_2(\mathbb{N}\setminus C_k,n)$ for all $n\geq 2N_k-1$.

Suppose that $\frac{14}{3}N_k-1=14\cdot4^k-1=a_k+a_k'$ for some $a_k,a_k'\in C_k,~a_k<a_k'$. Then there exist integers $a_{k-1},a_{k-1}'\in C_{k-1}$ such that
$$14\cdot4^k-1=4a_{k-1}+(4a_{k-1}'+3)$$
or there exist integers $a_{k-1},a_{k-1}'\in \mathbb{N}\setminus C_{k-1}$ such that
$$14\cdot4^k-1=(4a_{k-1}+1)+(4a_{k-1}'+2).$$
Hence,
$$\frac{14}{3}N_{k-1}-1=14\cdot4^{k-1}-1=a_{k-1}+a_{k-1}',$$
where $a_{k-1},a_{k-1}'\in C_{k-1}$ or $a_{k-1},a_{k-1}'\in \mathbb{N}\setminus C_{k-1}$. Clearly, $a_{k-1}\neq a_{k-1}'$. Since $R_2(C_{k-1},n)=R_2(\mathbb{N}\setminus C_{k-1},n)$ for all $n\geq 2N_{k-1}-1$, we have $R_2(C_{k-1},\frac{14}{3}N_{k-1}-1)\geq 1$, a contradiction. This completes the proof.
\end{proof}

\begin{proof}[Proof of Theorem \ref{thm3}]
For odd integers $N$, let
$$A\cap[0,2N-1]=\left\{1,3,5,\cdots,N-2,N,N+1,\cdots,\frac{3N-1}{2}\right\}$$
and let $2a\in A,~2a+1\notin A$ if $a\in A$ and $2a\notin A,~2a+1\in A$ if $a\notin A$ for $a\geq N$. Then $|A\cap A_0|=+\infty$, $|A\cap B_0|=+\infty$ and
$$A\cap[2N,3N-1]=\{2N,2N+2,2N+4,\cdots,3N-1\}.$$
It is easy to see that $R_2(A,3N-1)=0$. By Theorem A, $R_2(A,n)=R_2(\mathbb{N}\setminus A,n)$ for all $n\geq 2N-1$.

For even integers $N$, let
$$A\cap[0,2N-1]=\left\{0,2,4,\cdots,N-2,N,N+1,\cdots,\frac{3N}{2}-1\right\}$$
and let $2a\in A,~2a+1\notin A$ if $a\in A$ and $2a\notin A,~2a+1\in A$ if $a\notin A$ for $a\geq N$. Similar to the above argument, we can get the same result. This completes the proof.
\end{proof}

\section*{Acknowledgements}
The first author was supported by the National Natural Science Foundation of China, Grant No. 11771211 and the Project of Graduate Education Innovation of Jiangsu Province, Grant No. KYCX19\_0775. The second author was supported by the NKFIH Grant No. K129335 and by the Lend\" ulet program
of the Hungarian Academy of Sciences (MTA), under grant number LP2019-15/2019. The third author was supported by the National Natural Science Foundation for Youth of China, Grant No. 11501299, the Natural Science Foundation of Jiangsu Province, Grant Nos. BK20150889, 15KJB110014 and the Startup Foundation for Introducing Talent of NUIST, Grant No. 2014r029.

\renewcommand{\refname}{Bibliography}

\end{document}